\documentclass[a4paper,11pt]{article}

\usepackage{amsmath,amsfonts,amsthm}

\newtheorem{theorem}{Theorem}[section]
\begin{document}
\sloppy

\begin{center}
{\large \textbf{Weak and strong type estimates 
for fractional integral operators on Morrey spaces
in metric measure spaces}}\\[0pt]

\vspace{1cm} {\sc I. Sihwaningrum and Y. Sawano$^*$}

\end{center}

\bigskip

\begin{abstract}
\noindent We discuss here a weak and strong type estimate for fractional
integral operators on Morrey spaces,
where the underlying measure $\mu$ does not always satisfy
the doubling condition.

\medskip


\medskip

\end{abstract}
%%%%%%%%%%%%%%%%%%%%%%%%%%%%%%%%%%%%%%%%%%%%%%%%%%%%%%%%%%%%%%%%%%%%%%%%%%%%
%

\bigskip

\section{Introduction}

The aim of this paper is to propose
a framework of Morrey spaces and fractional integral operators
when we are given a Radon measure $\mu$
on a metric measure space $(X,d,\mu)$,
where $\mu$ is a Radon measure.

We recall that the Riesz potential $I_\alpha$ on ${\mathbb R}^d$
is given by
\[
I_\alpha f(x)=\int_{{\mathbb R}^d}
\frac{f(y)}{|x-y|^{d-\alpha}}\,dy.
\]
According to the Hardy-Littlewood-Sobolev theorem
\cite{4,5,10},
$I_\alpha$ is bounded
from $L^p({\mathbb R}^d)$ to $L^q({\mathbb R}^d)$
as long as $p,q \in (1,\infty)$ satisfy
$\frac1q=\frac1p-\frac{\alpha}{d}$.
Morrey spaces, named after C.~Morrey,
can also be used to describe the boundedness property
of $I_\alpha$.
Here we adopt the following notation
to denote Morrey spaces.
Let $1 \le q \le p<\infty$.
For a measurable function $f$ on ${\mathbb R}^d$,
we define
\[
\|f\|_{{\mathcal M}^p_q}
:=
\sup\left\{|B|^{\frac1p-\frac1q}\|f\|_{L^q(B)}
\,:\,B\mbox{ is a ball }\right\}.
\]
The space ${\mathcal M}^p_q({\mathbb R}^d)$
denotes the set of all measurable functions $f$
for which the norm
$\|f\|_{{\mathcal M}^p_q}$ is finite.
According to Adams \cite{Ad},
$I_\alpha$ is bounded from ${\mathcal M}^p_q({\mathbb R}^d)$
to ${\mathcal M}^s_t({\mathbb R}^d)$,
provided that $p,q,s,t \in (1,\infty)$ satisfy
$\frac{p}{q}=\frac{t}{s}, \, \frac{1}{q}=\frac{1}{p}-\frac{\alpha}{d}$.

In this paper,
we aim to show that this theorem is independent
from the geometric structure of ${\mathbb R}^d$
by extending it to metric measure spaces,
where all we have are
the distance function $d$
and the Radon measure $\mu$.

Let $(X, d, \mu)$ be a metric measure space
with a distance function $d$ and a Borel measure $\mu$.
Recall that the measure $\mu$ is a doubling measure
if it satisfies the so-called {\it doubling condition},
that is, there exists a constant $C>0$ such that
\begin{equation}\label{E1}
    \mu (B(a,2r)) \le C\mu (B(a,r))
\end{equation}
for every ball $B(a,r)$ with center $a\in X$ and radius
$r>0$.
The doubling condition was a key property
in classical harmonic analysis
but around a decade ago,
it turned out to be unnecessary.
The point is that we modify the related definitions.
Indeed, 
in the present paper,
we propose to redefine the fractional integral operator
by
\begin{equation}\label{E3}
    I_\alpha f(x)
:=
\int_X \frac{f(y)}{\mu(B(x,2d(x,y)))^{1-\alpha}}\;d\mu(y).
\end{equation}
Note that the definition is independent
of any notion of dimensions.
The same can be said for Morrey spaces,
which we define now.
For $k>0, 1\le p<\infty$ and $f\in L^1_{\textrm{loc}}(\mu)$,
the norm is given by
\begin{align*}
\lefteqn{
\|f\|_{\mathcal{M}^p_1(k,\mu)}
}\\
&:=\sup
\left\{
\mu(B(x,k r))^{1/p-1}\|\chi_{B(x,r)}f\|_{L^1(\mu)}
\,:\,x \in X, \, r>0, \, \mu(B(x,r))>0
\right\},
\end{align*}
where $\chi_{B(x,r)}$ denotes
the characteristic function of the ball $B(x,r)$.

We will prove here that $I_\alpha$ satisfies
weak and strong type estimates 
on Morrey spaces.
Our main results are:
\begin{theorem}\label{T3}
If $1<p<\infty$, $1<s<\infty$, $0<\alpha<\frac{1}{p}$ 
and $\frac{1}{s}=\frac{1}{p}-\alpha$, then there exists $C>0$ such that
\[
\mu\{x\in B(a,r):I_\alpha f(x)>\gamma \}\le C\mu (B(a,6r))^{1-1/p}\left(\frac{\|f\|_{\mathcal{M}^p_1(2,\mu)}}{\gamma}\right)^{s/p}
\]
for all positive $\mu$-measurable functions $f$.
\end{theorem}

\begin{theorem}\label{T6}
If $1<q \le p<\infty$, $1<s<\infty$, $0<\alpha<\frac{1}{p}$,
$\frac{q}{p}=\frac{t}{s}$ and $\frac{1}{s}=\frac{1}{p}-\alpha$, 
then there exists $C>0$ such that
\[
\|I_\alpha f\|_{\mathcal{M}^s_t(6,\mu)}
\le C
\|f\|_{\mathcal{M}^p_q(2,\mu)}
\]
for all positive $\mu$-measurable functions $f$.
\end{theorem}

It hardly looks likely
to replace $2d(x,y)$ with $d(x,y)$
in the definition of fractional integral operators
and have the similar results
according to the example in
\cite[Section 2]{SS-colm}.
The proof is a future work.

\smallskip

\section{Main Results}

\noindent
We define, for $k>0$,
the centered maximal operator
\[
    M_kf(x)
:=
\sup_{r>0}\frac{1}{\mu(B(x,kr))}\int_{B(x,r)}|f(y)|\; d\mu (y)
\quad (x \in {\rm supp}(\mu)).
\]
For the maximal operator $M_2$,
we prove the following boundedness property
on Morrey spaces.
\begin{theorem}\label{T1}
For any $\gamma >0$,
any positive $\mu$-measurable function
$k$ 
and 
any ball $B(a,r)$,
\[
\mu\{x\in B(a,r):M_2f(x)>\gamma\}
\le
4\frac{\mu(B(a,6r))^{1-1/p}}{\gamma}
\|f\|_{\mathcal{M}^p_1(2,\mu)}.
\]
\end{theorem}

\begin{proof}
We actually prove
\begin{equation}\label{eq:120702-1}
\mu\{x\in B(a,r):M_2f(x)>2\gamma\}
\le 2
\frac{\mu(B(a,6r))^{1-1/p}}{\gamma}
\|f\|_{\mathcal{M}^p_1(2,\mu)}.
\end{equation}
Once we prove
\begin{equation}\label{eq:120702-2}
\mu\{x\in B(a,r):M_2[\chi_{B(a,3r)}f](x)>\gamma\}
\le
\frac{\mu(B(a,6r))^{1-1/p}}{\gamma}
\|f\|_{\mathcal{M}^p_1(2,\mu)}
\end{equation}
and
\begin{equation}\label{eq:120702-3}
\mu\{x\in B(a,r):M_2[\chi_{X \setminus B(a,r)}f](x)>\gamma\}
\le
\frac{\mu(B(a,6r))^{1-1/p}}{\gamma}
\|f\|_{\mathcal{M}^p_1(2,\mu)},
\end{equation}
then estimate (\ref{eq:120702-1})
follows automatically.
Estimate (\ref{eq:120702-2})
follows from the weak-$L^1(\mu)$ boundedness
of $M_2$ (see \cite{S1,T}).

Denote by ${\mathcal B}(\mu)$
the set of all balls with positive $\mu$-measure.
A geometric observation shows that
\[
M_2[\chi_{X \setminus B(a,3r)}f](x)
\le
\sup_{\substack{B\in {\mathcal B}(\mu), \, B \cap B(a,r) \ne \emptyset,
\\ B \cap (X \setminus B(a,3r))\ne \emptyset}}
\frac{1}{\mu(2B)}\int_B|f(y)|\,d\mu(y).
\]
Let $B$ be a ball which intersects
both $B(a,r)$ and $X \setminus B(a,3r)$.
The ball $B$ engulfs $B(a,r)$
if we double the radius of $B$.
Thus,
\begin{align*}
\lefteqn{
\mu(B(a,6r))^{1/p-1}
\mu\{x\in B(a,r):M_2[\chi_{B(a,3r)}f](x)>\gamma\}
}\\
&\le
\mu(B(a,r))^{1/p-1}
\sup_{\substack{B\in {\mathcal B}(\mu), \, B \cap B(a,r) \ne \emptyset,
\\ B \cap (X \setminus B(a,3r))\ne \emptyset}}
\frac{1}{\mu(2B)}\int_B|f(y)|\,d\mu(y)\\
&\le
\sup_{\substack{B\in {\mathcal B}(\mu), \, B \cap B(a,r) \ne \emptyset,
\\ B \cap (X \setminus B(a,3r))\ne \emptyset}}
\frac{\mu(2B)^{1/p}}{\mu(2B)}\int_B|f(y)|\,d\mu(y)\\
&\le
\|f\|_{{\mathcal M}^p_1(2,\mu)}.
\end{align*}
Thus, (\ref{eq:120702-3}) follows.
\end{proof}

Analogously, the following inequality holds:
\begin{theorem}\label{T7}
Let $1<q \le p<\infty$.
Then there exists $C>0$ such that
\[
\|M_2f\|_{\mathcal{M}^p_q(6,\mu)}
\le C
\|f\|_{\mathcal{M}^p_q(2,\mu)}
\]
for all positive $\mu$-measurable functions.
\end{theorem}
The proof of Theorem \ref{T7}
being similar to that of Theorem \ref{T1},
we skip the proof,
which is based on the $L^q(\mu)$-boundedness 
of $M_2$ established in \cite{S1}.

Next, we prove a Hedberg type estimate \cite{6}.

\begin{theorem}\label{T2}
If $1<p<\infty$ and $0<\alpha<\frac{1}{p}$,
then there exists $C>0$ such that
\[
|I_\alpha f(x)|
\le C
M_2f(x)^{1-p\alpha}\|f\|_{\mathcal{M}^p_1(2,\mu)}^{p\alpha}
\quad (x \in X)
\]
for all positive $\mu$-measurable functions.
\end{theorem}

\begin{proof}
Let $x\in X$ be fixed. We define
\[
R_k(x):=\inf\left(\{R>0:\mu(B(x,2R))>2^k\}\cup \{\infty\} \right).
\]
Then, we have
\begin{align*}
\lefteqn{
    |I_\alpha f(x)|
}\\
&\le 
\sum_{k=-\infty}^{\infty}\lim_{\varepsilon \downarrow 0}\int_{B(x,R_k(x))\setminus B(x,R_{k-1}(x))}
     \frac{|f(y)|}{\mu(B(x,2d(x,y)+\varepsilon))^{1-\alpha}}\;d\mu(y)\\
     &= \sum_{k=-\infty}^{\infty}\lim_{\varepsilon \downarrow 0}\int_{B(x,R_k(x))\setminus B(x,R_{k-1}(x))}
     \frac{|f(y)|}{\mu(B(x,2R_{k-1}(x)+\varepsilon))^{1-\alpha}}\;d\mu(y)\\
     &\le \sum_{k=-\infty}^{\infty}\lim_{\varepsilon \downarrow 0}
     \frac{1}{\mu(B(x,2R_{k-1}(x)+\varepsilon))^{1-\alpha}}\int_{B(x,R_k(x))\setminus B(x,R_{k-1}(x))}
     |f(y)|\;d\mu(y)\\
     &\le \sum_{k\in \mathbb{Z};R_{k-1}(x)<R_k(x)}
\lim_{\varepsilon \downarrow 0}
     \frac{1}{\mu(B(x,2R_{k-1}(x)+\varepsilon))^{1-\alpha}}
\int_{B(x,R_k(x))}|f(y)|\;d\mu(y).
\end{align*}
The condition $R_{k-1}(x)<R_k(x)$ means that
$$
2^{k-1}<\mu(B(x,2R_{k-1}(x)+\varepsilon))\le 2^k
$$ 
for each
$\varepsilon \in (0, R_k(x)-R_{k-1}(x))$. Therefore
\begin{align*}
    |I_\alpha f(x)|
     &\le C \sum_{k=-\infty}^{\infty} 2^{k\alpha} \min \left( M_2f(x),2^{-k/p}\|f\|_{\mathcal{M}^p_1(2,\mu)}\right)\\
     &\le C  Mf(x)^{1-p\alpha}\|f\|_{\mathcal{M}^p_1(2,\mu)}^{p\alpha}.
\end{align*}
Thus, the estimate is proved.
\end{proof}

Now we prove Theorem \ref{T3}.

\begin{proof}
For $|I_\alpha f(x)|>\gamma$, Theorem \ref{T2} gives us
\[
M_2f(x)>
\left( \frac{\gamma}{C\|f\|_{\mathcal{M}^p_1(2,\mu)}^{p\alpha}}
\right)^{1/(1-p\alpha)}.
\]
Hence, by applying Theorem \ref{T1}, we obtain
\begin{align*}
\lefteqn{
    \mu\{x\in B(a,r):|I_\alpha f(x)|>\gamma \}
}\\
&\le \mu\left\{x\in B(a,r):M_2f(x)>\left( \frac{\gamma}{C\|f\|_{\mathcal{M}^p_1(2,\mu)}^{p\alpha}} \right)^{1/(1-p\alpha)} \right\}\\
    & \le C\mu (B(a,6r))^{1-1/p}\|f\|_{\mathcal{M}^p_1(2,\mu)}
\left(\frac{\|f\|_{\mathcal{M}^p_1(2,\mu)}^{p\alpha}}{\gamma}\right)^{1/(1-p\alpha)}\\
    &\le C\mu (B(a,6r))^{1-1/p}
\frac{\|f\|_{\mathcal{M}^p_1(2,\mu)}^{1+\alpha s}}{\gamma ^{s/p}}\\
     & \le C\mu (B(a,6r))^{1-1/p}
\left(\frac{\|f\|_{\mathcal{M}^p_1(2,\mu)}}{\gamma}\right)^{s/p}.
\end{align*}
Thus, the proof is complete.
\end{proof}

Theorem \ref{T6} can be proved in a similar way
by using Theorem \ref{T2}.
\bigskip

\noindent{\bf Acknowledgments}. The first
author 
was supported by Fundamental Research Program 2012
by Directorate General of Higher Education, Ministry of Education and Culture, 
Indonesia.
The second author was financially supported 
by Grant-in-Aid for Young Scientists (B), No. 21740104,
Japan Society for the Promotion of Science.

This research project is supported
by the GCOE program of Kyoto University.

\bigskip

\noindent {\sc Idha SIHWANINGRUM}\\
Faculty of Sciences and Engineering\\
Jenderal Soedirman University\\ Purwokerto, 53122 Indonesia\\
email: idha.sihwaningrum@unsoed.ac.id

\medskip

\noindent {\sc Yoshihiro Sawano}\\
Department of Mathematics and Information Sciences\\
Tokyo Metropolitan University\\
Tokyo 192-0397, Japan \\
email: ysawano@tmu.ac.jp\\

Received:

\end{document}